\documentclass[12pt]{article}

\usepackage[ansinew]{inputenc}

\usepackage{amsfonts,amssymb,amsmath,amsthm,graphicx}

\newtheorem{theorem}{Theorem}
\newtheorem{lemma}[theorem]{Lemma}
\newtheorem{corollary}[theorem]{Corollary}

\newtheorem{prop}[theorem]{Proposition}
\newtheorem{remark}[theorem]{Remark}
\newcommand{\nn}{\nonumber}
\newcommand{\na}[1]{\noalign{\noindent #1}}
\renewcommand{\ni}{\noindent}
\newcommand{\R}{\ensuremath{\mathbb{R}}}
\newcommand{\C}{\ensuremath{\mathbb{C}}}
\newcommand{\Z}{\ensuremath{\mathbb{Z}}}
\newcommand{\imag}[1]{\ensuremath{\mathrm{Im}(#1)}}
\newcommand{\real}[1]{\ensuremath{\mathrm{Re}(#1)}}
\newcommand{\floor}[1]{\ensuremath{\lfloor #1\rfloor}}
\newcommand{\helmn}{\ensuremath{\left(\Delta +
      k^{2}n^{2}\right)}}
\newcommand{\helm}{\ensuremath{\left(\Delta +
      k^{2}\right)}}
\newcommand{\bdry}{\partial}
\newcommand{\ddnu}[1]{\ensuremath{\frac{\partial
      #1}{\partial\nu}}}
\newcommand{\restricted}[1]{\ensuremath{\Big|_{#1}}}

\renewcommand{\bar}[1] {\ensuremath{\overline{#1}}}

\newcommand{\halfL}{\ensuremath{{\frac{L}{2}}}}
\newcommand{\half}{\ensuremath{{\frac{1}{2}}}}

\newcommand{\espace}{\vspace*{2mm}}

\newcommand{\jm} {\ensuremath{\frac{j}{m}}}
\newcommand{\jpm} {\ensuremath{\frac{j+1}{m}}}
\newcommand{\jmm} {\ensuremath{\frac{j-1}{m}}}
\newcommand{\sjm} {\ensuremath{S_{\frac{j}{m}}}}
\newcommand{\sjmm} {\ensuremath{S_{\frac{j-1}{m}}}}
\newcommand{\ijm} {\ensuremath{I_{\frac{j}{m}}}}
\newcommand{\ijmm} {\ensuremath{I_{\frac{j-1}{m}}}}
\newcommand{\Pmp} {\ensuremath{p_{m}^{o}}}
\newcommand{\Pmm} {\ensuremath{p_{m}^{e}}}
\newcommand{\Pm} {\ensuremath{p_{m}}}
\newcommand{\pb} {\ensuremath{p_{\beta}}}
\newcommand{\jpfjm} {\ensuremath{j+\floor{\jm}}}

\numberwithin{equation}{section}

\begin{document}

\title{Transmission Eigenvalues in One Dimension}

\author{{John Sylvester
\thanks{Department of Mathematics,University of Washington,
 Seattle, Washington 98195,USA. (sylvest@uw.edu). Research
 was supported in part by NSF grant  DMS-1007447}}}

 \date{}

\maketitle

\thispagestyle{empty}

\begin{abstract}
We show how to locate all the transmission eigenvalues for
a one dimensional constant index of refraction on an interval.
\end{abstract}

\section{Introduction}

The scattering operator for the time harmonic Helmholtz
equation with compactly supported index of refraction
\(n(x)\) maps the asymptotics of incident waves (Herglotz
wave functions) \(v_{0}\) to the asymptotics of (outgoing)
scattered waves \(u^{+}\) . Both \(u^{+}\) and \(v^{0}\)
are defined in \(\R^{n}\) and satisfy
\begin{eqnarray}
  \nn
 & \helmn u^{+} = (1-n^{2})k^{2}v_{0}&
\\\nn
&\helm v_{0} = 0&
\end{eqnarray}
in  all of \(\R^{n}\). 

If the scattering operator has a zero eigenvalue, we say
that \(k^{2}\) is a transmission eigenvalue
\cite{MR983987}, \cite{MR934695}. It follows from
Rellich's theorem and unique continuation that the
scattered wave \(u^{+}\) is identically zero outside any
domain \(D\) that contains the support of \(n\), and
therefore that  \(u=u^{+}\) and  \(v=v^{0}\) are a pair of
nontrivial functions satisfying
\begin{eqnarray}
  \label{eq:28}
 & \helmn u = (1-n^{2})k^{2}v&\qquad \mathrm{in }\ D
\\\label{eq:7}
&\helm v = 0& \qquad \mathrm{in }\ D
\\\label{eq:31}
&u\restricted {\bdry D} = 0 \qquad
\ddnu{u}\restricted{\bdry D}= 0&
\end{eqnarray}

Whenever such a nontrivial pair exists, we say that
 \(k^{2}\) is an interior transmission
eigenvalue for the pair \((D,n)\).\footnote{See  section 8
  of \cite{MR2986407}, as well as \cite{MR2438784}, for
  more details about the connection between transmission
  eigenvalues and interior transmission eigenvalues}
\begin{remark}
  If we scale  \(v\) by introducing  \(w = k^{2}v\) , then the
  \(k^{2}\) on the right hand side of equation
  (\ref{eq:28}) disappears, and we see that the interior
  transmission eigenvalue problem is a true generalized
  eigenvalue problem which can be written as 
  \begin{eqnarray}
    \nn
    \Delta u - (1-n^{2})w &=& -k^{2}n^{2}u
    \\\nn
    \Delta w             &=& -k^{2}w
  \end{eqnarray}
with the same boundary conditions as in (\ref{eq:31}) \cite{MR2888291}.
\end{remark}
The analogous interior problem for obstacle scattering is
the Dirichlet problem,
\begin{eqnarray}
  \nn
  &\helm v = 0& \qquad \mathrm{in }\ D
\\\nn
&v\restricted {\bdry D} = 0&
\end{eqnarray}
and the Dirichlet eigenvalues are fairly well understood. In
particular, the Dirichlet eigenvalues of a one dimensional
interval of length  \(L\)  are exactly
\(\frac{n^{2}\pi^{2}}{L^{2}}\). The one dimensional interior
transmission eigenvalue problem, even for a constant index
of refraction, is not as simple.  In one dimension, with
the index of refraction  \(n(x)\) equal to the positive constant
\(\sigma\), equations (\ref{eq:28})~(\ref{eq:7})~(\ref{eq:31}) 
become
\begin{eqnarray}
  \label{eq:1}
  &u'' + k^{2}\sigma^{2}u = k^{2}(1-\sigma^{2})v&
\\ \label{eq:2}
&v'' + k^{2} v = 0&
\\\label{eq:3}
&u(\pm\halfL)= 0 \qquad u'(\pm\halfL) = 0&
\end{eqnarray}
so that  \(k^{2}\) is an interior transmission eigenvalue if
and only there exist a nontrivial pair of eigenfunctions
\((u,v)\) satisfying (\ref{eq:1})(\ref{eq:2})(\ref{eq:3}).\\

We will show in proposition \ref{sec:analyticEquation}
that \(k^{2}\) is an interior transmission eigenvalue
if and only if \(k\) is a root of  the 
equation
\begin{eqnarray}
  \label{eq:6}
  (\sigma-1)^{2}\sin^{2}((\sigma+1)\frac{kl}{2}) -
  (\sigma+1)^{2}\sin^{2}((\sigma-1)\frac{kl}{2}) = 0
\end{eqnarray}
and the two theorems below describe the roots of
(\ref{eq:6}), which we will call the interior transmission
wavenumbers (positive square roots of interior transmission eigenvalues).

\begin{theorem}\label{sec:imaginaryK}
  The interior transmission wavenumbers satisfy
  \begin{eqnarray}
    \nn
    |\imag{k}| \le \frac{2}{L}\log{\frac{3\sigma+1}{|\sigma-1|}}
  \end{eqnarray}
\end{theorem}
\goodbreak

In theorem \ref{sec:itw-main} below, the notation 
\(\floor{\jm} \) denotes the \textit{floor}, the greatest integer
less than or equal to  \(\frac{j}{m}\). The condition    
\(\floor{\frac{j}{m}}=\floor{\frac{j+1}{m}}\) means that
there is no integer between \jm\ and \jpm, and  \(\Z\)
denotes the integers.\\

\begin{theorem}\label{sec:itw-main}
 
Let \(m=\frac{\sigma+1}{\sigma-1}\), and let  \(j\in\Z\)
  \begin{enumerate}
  \item If  \(\frac{j}{m}\not\in\Z\) and
    \(\frac{j+1}{m}\not\in\Z\), there are exactly two simple
    interior transmission wavenumbers in the strip
    \begin{eqnarray}
      \nn
      \frac{2j\pi}{(\sigma+1)L}<\real{k}<\frac{2(j+1)\pi}{(\sigma+1)L}
    \end{eqnarray}
    \begin{enumerate}
    \item If \(\floor{\frac{j}{m}}\ne\floor{\frac{j+1}{m}}\),
        both  are real
      \item If
        \(\floor{\frac{j}{m}}=\floor{\frac{j+1}{m}}\) they
         are complex conjugates with nonzero imaginary parts
    \end{enumerate}

  \item If \(\frac{j}{m}\in\Z\), there is a quadruple
    interior transmission wavenumber
    at \(k=\frac{j}{m}\) and no others in the strip
    \begin{eqnarray}
      \nn
      \frac{2(j-1)\pi}{(\sigma+1)L}<\real{k}<\frac{2(j+1)\pi}{(\sigma+1)L}
         \end{eqnarray}
  \end{enumerate}
\end{theorem}

\goodbreak
The transmission wavenumbers for two  \(\sigma\)'s, with  \(L=2\),  are
plotted below. The vertical blue lines indicate the sets
\(\real{k} = \frac{2j\pi}{(\sigma+1)L}\).  The roots that
fall on these lines are quadruple.

\begin{center}

\includegraphics[width=\linewidth]{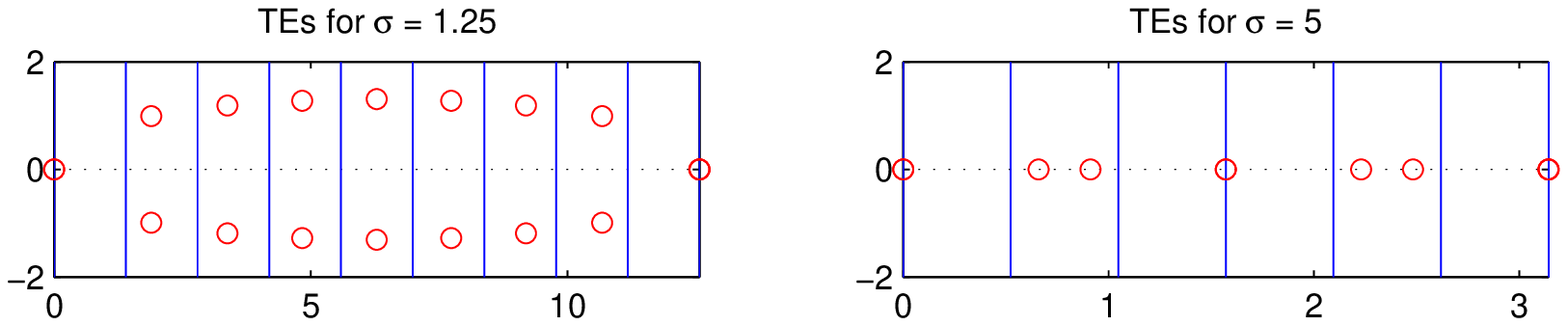}
\end{center}

As these  \(\sigma\)'s are rational, the transmission
eigenvalues are periodic. Remark \ref{sec:remark}
below explains how we did the computation.

\section{Lemmas and Proofs}
\label{sec:algebraic-equation}

Because we are dealing with constant coefficients, the
transmission wavenumbers  \(k\) are the zeros of the analytic
function on the left hand side of equation (\ref{eq:6}). 
This analytic function has two factors, and our proofs
simplify if we treat each factor separately. 

\begin{prop}\label{sec:analyticEquation}
  \(k\) is a transmission wavenumber if, and only if,  \(k\) satisfies
  \begin{eqnarray}
    \label{eq:13}
    (\sigma-1)\sin((\sigma+1)\frac{kl}{2})
    = (\sigma+1)\sin((\sigma-1)\frac{kl}{2})
\\\na{or}\label{eq:14}
   (\sigma-1)\sin((\sigma+1)\frac{kl}{2})
   = -(\sigma+1)\sin((\sigma-1)\frac{kl}{2})
  \end{eqnarray}
  The wavenumber \(k\) satisfies (\ref{eq:13}) if there
  exists an eigenfunction pair \((u,v)\) of odd functions
  of \(x\), and \(k\) satisfies (\ref{eq:14}) if there
  exists an eigenfunction pair \((u,v)\) of even
  functions.
\end{prop}
\begin{proof}
If a pair \(u,v\) satisfies (\ref{eq:1}-\ref{eq:3}), then so do
their odd and even parts, and the calculation is a little
simpler for each of these parts. The odd parts satisfy
(\ref{eq:1}-\ref{eq:2}) with (\ref{eq:3}) replaced by 
\begin{eqnarray}
  \nn
  &u(0)= v(0) = 0 \qquad u(\halfL) = u'(\halfL) = 0
\end{eqnarray}
and the even parts satisfy (\ref{eq:1}-\ref{eq:2}) with
(\ref{eq:3}) replaced by 
\begin{eqnarray}
  \nn
  &u'(0)= v'(0) = 0 \qquad u(\halfL) = u'(\halfL) = 0
\end{eqnarray}
In the odd case, the boundary conditions at zero imply
that
\begin{eqnarray}
  \nn
  &v = \sin(kx)&
  \\\na{ and}
  &u = -\sin(kx) + A\sin(k\sigma x)
\end{eqnarray}
for some constant \(A\). The boundary conditions
at \(x=\halfL\) imply
 that

\begin{eqnarray}
  \label{eq:8}
  \sin(\sigma k\halfL)\cos(k\halfL) =
          \sigma\sin(k\halfL)\cos(\sigma k\halfL)
\end{eqnarray}
 Noting that the left and right
hand sides are the \textit{beats} representation of the
combination of the frequencies \((\sigma+1)\halfL\) and
\((\sigma-1))\halfL\) suggests that we rewrite (\ref{eq:8}) as
\begin{eqnarray}
  \nn
  (\sigma-1)\sin((\sigma+1))\halfL k) = (\sigma+1)\sin((\sigma-1)\halfL k)
\end{eqnarray}
which is (\ref{eq:13}). A similar calculation shows the
existence of even transmission eigenfunction pairs if and
only if \(k\) satisfies (\ref{eq:14}).
\end{proof}

If we introduce the new variables
\begin{eqnarray}
  \label{eq:22}
  m = \frac{\sigma+1}{\sigma-1}
  \\\na{and}\nn
  z=(\sigma-1)k\frac{L}{2\pi}
\end{eqnarray}
then equations (\ref{eq:13}) and (\ref{eq:14}) become
equations for a  \(z\) whcih depend on
the single real parameter \(m\). They become
\begin{eqnarray}
\label{eq:10}
 p_{m}^{o}(z) := \sin(m\pi z) - m\sin(\pi z) = 0
\\\label{eq:11}
 p_{m}^{e}(z) := \sin(m\pi z) + m\sin(\pi z) = 0
\end{eqnarray}
where the superscripts refer to even and odd, respectively.
Because \(\sigma>0\), \(|m|>1\).  Because equations
(\ref{eq:10}) and (\ref{eq:11}) remain the same if we
change the sign of \(m\), we may also assume that
\(m>1\), and will do so throughout\footnote{Alternatively,
  we could start with the assumption that  \(\sigma>1\),
  and then use the relationship between the transmission
  wavenumbers of  \(\sigma\) and those of  \(\frac{1}{\sigma}\),  
  \(k_{j}(\frac{1}{\sigma})=\sigma k_{j}(\sigma)\) to
  show that our formulas and estimates apply to
  \(\sigma<1\) as well.}.
\begin{remark}\label{sec:remark}
  If  \(m=\frac{p}{q}\) is rational and we set  \(w =
  e^{\frac{i\pi z}{q}}\), then (\ref{eq:10}) becomes a
  polynomial equation in  \(w\). 
  \begin{equation}
    \label{eq:9}
    w^{2p} -1 = \frac{p}{q}\left(w^{p+q} - w^{p-q}\right)
  \end{equation}
  We can use a polynomial root finding algorithm to find
  the solutions  and then take logarithms and rescale
  to find the interior transmission wavenumbers
  \(k\). This is how we produced figure 1. 
\end{remark}

  With the notation introduced in
(\ref{eq:22}), theorem \ref{sec:imaginaryK} becomes

\begin{theorem}
  If  \(z\) satisfies (\ref{eq:10}) or (\ref{eq:11}), then
  \begin{eqnarray}\label{eq:21}
        |\imag{z}|\le\frac{\log(2m+1)}{\pi(m-1)}    
  \end{eqnarray}
\end{theorem}
\begin{proof}
If \(z=x+iy\) satisfies (\ref{eq:10})
  \begin{eqnarray}
      \nn
      e^{m\pi |y|}- e^{-m\pi |y|}&\le&2|\sin m\pi z| = 2m|\sin\pi z| \le
      me^{\pi|y|}+ me^{-\pi|y|}
      \\\na{\noindent so that}\nn
       e^{m\pi|y|}&\le&  e^{-m\pi|y|} + me^{\pi|y|}+ me^{-\pi|y|}
       \\\na{\noindent Dividing both sides by  \(e^{\pi|y|}\) gives\espace}\nn
       e^{(m-1)\pi|y|}&\le&e^{-(m+1)\pi|y|} + m+ me^{-2\pi|y|}\le 2m+1
    \end{eqnarray}
Now taking the logarithm of both sides yields (\ref{eq:21})
\end{proof}

Locating the real parts of the roots of \(p_{m}(z) =
p_m^{o}(z)p_{m}^{e}(z)\) will take more work.  We will
show (roughly) that there are always two roots in certain of
strips of the complex plane. We define

 \begin{eqnarray}
   \nn
   I_{\jm} &=& \left\{x\ \Big\vert\ \ \jm<x<\jpm\right\}
\\\nn
    S_{\jm}&=& \left\{z\ \Big\vert\ \ \jm<\real{z}<\jpm\right\}
 \end{eqnarray}

A restatement of theorem \ref{sec:itw-main}, using the
notation of equation (\ref{eq:22}), is theorem
\ref{sec:roots-p_m} below.

 \begin{theorem}\label{sec:roots-p_m}
   Let  \(m>1\).
   \begin{enumerate}
   \item If neither \jm\  nor \jpm\  are integers, then
     \(p_{m}\) has two simple roots in the strip  \sjm\ . 
     \begin{enumerate}
     \item If the interval \ijm\ contains an integer  \(k\),  both
       roots are real, and one is bigger and the other
       smaller than  \(k\).
     \item If the interval \ijm\ does not contain an
       integer, the roots are complex conjugates (with
       nonzero imaginary parts).
     \end{enumerate}
   \item If \jm\ is an integer, \jm\ is a quadruple root
     of \Pm\ , and \Pm\ has no other roots in \sjm or \sjmm.
   \end{enumerate}
 \end{theorem}

Theorem \ref{sec:roots-p_m} is an immediate consequence of the
analogous theorems for \Pmp\ and \Pmm\ . These theorems
are a little more complicated to state, but easier to prove.

\begin{theorem}\label{sec:roots-p_m-p}
   Let  \(m>1\).
   \begin{enumerate}
   \item Suppose that neither \jm\  nor \jpm\  are integers. 
     \begin{enumerate}
     \item\label{item:6} If the interval \ijm\ contains an
        integer \(k\) and \jpfjm is even, then \Pmp\ has a
        real root in the interval \((\jm,k)\) and no other
        roots in \sjm\, and \Pmm\ a root in \((k,\jpm)\)
        and no other roots in \sjm . If the interval \ijm\ contains an
        integer \(k\) and \jpfjm is odd, then \Pmp\ has a
        real root in the interval \((k,\jpm)\) and no other
        roots in \sjm\, and \Pmm\ a root in \((\jm,k)\)
        and no other roots in \sjm.
     \item\label{item:7}  If  the interval  \ijm\  does  not contain  an
       integer  and  \jpfjm is  odd,  \Pmp\  has 2  simple
       complex  conjugate  roots  in  \sjm and  \Pmm\  has
       none. If \jpfjm is  even, \Pmm\ has 2 simple complex
       conjugate roots in \sjm and \Pmp\ has none.
     \end{enumerate}
   \item\label{item:8} If \jm\ is an integer, neither \Pmp\ nor
     \Pmm\ have any
     roots in \sjm or \sjmm. If \jpfjm\ is even, \jm\ is a
     triple root of \Pmp and a simple root of \Pmm. If
     \jpfjm\ is odd, \jm\ is a triple root of \Pmm and a
     simple root of \Pmp.
   \end{enumerate}
 \end{theorem}

We will prove theorem \ref{sec:roots-p_m-p} (we will only
prove the results for  \Pmp)  by studying
a one parameter family of equations with  \(m\) fixed
and parameter  \(\beta\).  
\begin{eqnarray}
  \label{eq:12}
  p_{\beta}(z) = \sin(m\pi z) - \beta\sin(\pi z)
\end{eqnarray}

We will use the fact that that the roots of the entire
function \(p_{\beta}(z)\) depend continuously on the
parameter \(\beta\)\footnote{The continuity of the roots
  of analytic functions follows from the fact that roots
  are isolated, together with the argument principle.}. We
will also use the fact that, as
\(\beta\rightarrow\infty\), the roots of \(p_{\beta}\)
must approach those of \(\sin(\pi z)\). This last fact
follows from the first by multiplying (\ref{eq:12}) by
\(\alpha = \frac{1}{\beta}\), and the fact that the roots
of \(\alpha\sin(m\pi z) - \sin(\pi z)\) depend
continuously on \(\alpha\) as \(\alpha\) approaches zero.\\

Before giving the rigorous details, we give a rough
outline of the proof of Theorem \ref{sec:roots-p_m-p}. To
simplify this outline we ignore the case where \jm\ is an
integer.\\

\ni  Theorem \ref{sec:roots-p_m-p} makes two
assertions
\begin{enumerate}
\item It tells us the number of roots \Pmp\ in each strip
  \sjm. This number is the same for \pb\  for all  \(\beta>0\). 
  
\item It tells us the number of real roots in each strip
  when  \(\beta=m\). The number of real roots of \pb\
  changes as \(\beta\) increases from zero to  \(m\), but
  is the same for all  \(\beta\ge m\)  .
\end{enumerate}
\ni
We establish the number of roots in each strip by observing that:
\begin{enumerate}
  \item at  \(\beta=0\) all roots are real and on the
    boundary of the strips \sjm , and the derivatives are nonzero
    and pointing in the direction of the correct strip
    (lemma \ref{sec:roots-p_m-2}), so
    that, for small positive \(\beta\), the correct number of
    roots have entered each strip(corollary \ref{sec:lemma-small-beta}).
  \item Lemma \ref{sec:realsin} tells us that no roots can
    cross the boundary of \sjm, so the correct number of
    roots remain in each strip for all  \(\beta\)
    (corollary \ref{sec:roots-all-beta}). 
  \end{enumerate}

\noindent
We calculate the number of roots which are real at  \(\beta=m\) by
observing that:
  \begin{enumerate}
  \item As \(\beta\) approaches infinity, a single root
    must approach each integer (the roots of \(\sin \pi
    z\)) , and the other roots must leave every compact
    subset of \(\C\).

  \item A calculation(lemma \ref{sec:beta-roots-simple}) shows
    that for \(\beta\ge m\), all roots are simple.
 
  \item This means that any real roots of \pb\ in \sjm\ at
    \(\beta=m\) must remain real for \(\beta\ge m\) and
    remain in \sjm, so they remain in the interval
    \(I_{\jm}\).  But, as \(\beta\) approaches infinity,
    the roots of \pb\ must either approach the roots of
    \(\sin\pi z\) or leave every compact subset of
    \(\C\). The roots in the \sjm 's that don't contain an
    integer must leave every compact set, so they cannot
    remain in the interval \(I_{\jm}\), so they cannot
    have been real when \(\beta=m\).

  \item The \sjm's that contain an integer contain only
    one root for all \(\beta\), so the root remains simple
    and therefore real at \(\beta=m\). That root
    approaches the integer as \(\beta\) approaches
    infinity.
  \end{enumerate}

The detailed proof follows:

\begin{lemma}\label{sec:roots-p_m-2}
  The roots of \(p_{0}=\sin(m\pi z)\) are  \(\jm\). Each of
  these roots is simple, and thus continues to a unique
  root \(z_{j}(\beta)\) of  \(\pb\) for small
  \(\beta\). Moreover,
  \begin{eqnarray}
\frac{dz_{j}}{d\beta}\Big\vert_{\beta=0}&=&(-1)^{\jpfjm}|\sin(\frac{j\pi}{m})| 
  \end{eqnarray}
\end{lemma}

\begin{proof}
  At \(\beta=0\), the roots of \(\sin m\pi z\) are located
  at \jm\ and are all simple. The implicit function theorem implies that, for
  small \(\beta\), they remain simple and 
  depend analytically on \(\beta\), so we can define
  \(z_{j}(\beta)\) to be the continuation of that
  root. Differentiating (\ref{eq:12}) yields
  \begin{eqnarray}
    \nn
    &\frac{dz_{j}}{d\beta}\left(m\pi\cos(m\pi z_{j})
      +\beta\pi\cos(\pi z_{j})\right) = \sin(\pi z_{j})&
\\\na{so that}\nn
&\frac{dz_{j}}{d\beta}\Big\vert_{\beta=0} = 
\frac{\sin\frac{\pi j}{m}}{\cos\pi j}\quad
= (-1)^{j}\sin\frac{j\pi}{m}\quad
= (-1)^{\jpfjm}|\sin(\frac{j\pi}{m})|&
  \end{eqnarray}
\end{proof}

Knowledge of the derivatives tells us that the roots are
in the correct strips for small  \(\beta\).  

\begin{corollary}\label{sec:lemma-small-beta}
  For  \(\beta>0\) and sufficiently small, all roots of
  \pb\ are real and:
   \begin{enumerate}
   \item Suppose that neither \jm\  nor \jpm\  are integers. 
     \begin{enumerate}
     \item\label{item:1} If the interval \ijm\ contains an integer
       \(k\), then \pb\ has a real root in the interval
       \((\jm,k)\) and no other roots in \sjm\ if \jpfjm
       is even. \pb\ has
       a root in \((k,\jpm)\) and no other roots in \sjm\
       if \jpfjm is odd.  
     \item\label{item:2}  If  the interval  \ijm\  does  not contain  an
       integer  and  \jpfjm is  odd, \pb\ has no roots in
       \sjm. If \jpfjm is even, \pb\ has two roots in \sjm.
     \end{enumerate}
   \item\label{item:3} If \jm\ is an integer, \jm is a simple root of
     \pb. If \jpfjm\ is even, there is a simple root in
     each of \ijm\ and \ijmm. If \jpfjm\ is odd, there are
     no roots in \ijm\ or \ijmm.
   \end{enumerate}
\end{corollary}
\begin{proof}
  Because  \(\sin\bar{z} = \bar{\sin z}\), the roots of
  (\ref{eq:10}) must occur in conjugate pairs, and simple
  roots must remain real for small \(\beta\). Lemma
  \ref{sec:roots-p_m-2} tells us that, \(z_{j}\) starts at
  \jm\ and moves to the right if \jm\ is not an integer
  and \jpfjm\ is even, or to the left if   \jm\ is not an integer
  and \jpfjm\ is odd. If \ijm\ does not contain an
  integer, then   \(\jpm+\floor{\jpm}\) is odd whenever
  \jpfjm\ is even, and even whenever \jpfjm\ is odd, so
  that both  \(z_{j}\) and \(z_{j+1}\) move into \ijm\
  when \jpfjm\ is even, and away from \ijm\ when \jpfjm is
  odd, which establishes item \ref{item:2}.\\

  If \ijm\ contains an integer \(k\), then
   \(z_{j}\) moves into
  \ijm\ and \(z_{j+1}\) moves away from \ijm\ when
  \jpfjm\ is even, and \(z_{j_+1}\) moves into \ijm\ and
  \(z_{j}\) moves away when \jpfjm\ is odd. The root which
  moves into \ijm\ remains in either \((\jm,k)\) or
  \((k,\jpm)\) for small \(\beta\), which establishes item
  \ref{item:1}.\\

  Finally, if \jm\ is an integer, both terms in
  (\ref{eq:12}) vanish, so that \jm\ remains a root
  for all \(\beta\), and remains simple at least for
  \(\beta\) small. If \jm\ is an integer,
  \(\jmm+\floor{\jmm}= \jpfjm-2\) and
  \(\jpm+\floor{\jpm}=\jpfjm+1\), so 
  \(z_{j-1}\) moves into  \(I_{\jmm}\) and  \(z_{j+1}\)
  moves into \(I_{\jpm}\) if \jpfjm\ is even, and both
  move away from those intervals if \jpfjm is odd. This
  establishes item \ref{item:3} and finishes the proof.
\end{proof}

At a root of \pb(z), the real parts of the \(\sin m\pi z\)
must equal the real part of \(\beta\sin\pi z\). The lemma
below identifies the sets on which the real parts of these
terms vanish.

\begin{lemma}\label{sec:realsin}
   \(\real{\sin(z)}=0\) iff \(\sin(\real{z})=0\)  
\end{lemma}
\begin{proof}
  \begin{eqnarray}
    \nn
    \sin(x+iy) = \sin x\cosh y + i\cos x\sinh y
  \end{eqnarray}
and  \(\cosh y\) never vanishes. 
\end{proof}
A consequence of lemma \ref{sec:realsin} is that \pb\
  has no roots on \(\real{z} =\jm\) or on \(\real{z}=k\),
  when \(j\) or \(k\) are integers, unless \jm\ equals an
  integer. This is enough to show that the roots which
  entered each strip for small  \(\beta\), must remain
  there. 
\begin{corollary}\label{sec:roots-all-beta} For all \(0<\beta<\infty\)
  \begin{enumerate}
  \item\label{item:4} The conclusions in items \ref{item:1} and
    \ref{item:2} of corollary \ref{sec:lemma-small-beta},
    remain true.
  \item\label{item:5} If \jm\ is an integer and \jpfjm\ is
    odd, \jm\ is a simple root of \pb\ and \pb\ has no other
    roots in \sjm or  \(S_{\jpm}\) . If \jm\ is an integer and \jpfjm\ is
    even, \pb\ has three roots in \(S_{\jmm}\cup\sjm\),
    one of which is \jm.
  \end{enumerate}
\end{corollary}
\begin{proof}
  The hypothesis of item \ref{item:4}, that \jm\ and \jpm\
  are not integers, combine with lemma \ref{sec:realsin}
  to guarantee that \pb\ can have no roots on the boundary
  of \sjm. As the roots of the entire analytic function
  \(\pb\) depend continuously on  \(\beta\), this implies
  that no root  can enter or leave \sjm.
   Thus, if \sjm\ contained no roots for small
  \(\beta\), it contains no roots for all \(\beta\). If
  \sjm\ contained two roots for some \(\beta\), it contains
  two roots for all \(\beta\). This establishes item
  \ref{item:2} for all \(\beta\). To establish item
  \ref{item:1}, note that lemma \ref{sec:realsin} forces
  the simple root to remain in \(k<\real{z}<\jpm\) or
  \(\jm<\real{z}<k\). As the simple real root remains
  simple, it must also remain
  real.\\

The  two statements in item \ref{item:5} follow similarly
from the fact that no root can cross \jpm\ or \jmm, as
these cannot be integers if \jm\ is an integer because
\(m>1\). 
 \end{proof}
\begin{lemma}\label{sec:beta-roots-simple}
  For \(\beta>m\), all roots of \pb\ are simple. For
  \(\beta=m\), all roots are simple except \(z=\jm\) in
  the case that \(j+\frac{j}{m}\) is an even integer. In
  this case,\jm\ is a triple root of \pb.
\end{lemma}
\begin{proof}
  All roots of \pb\ that aren't real are simple, because
  there are no more than three in any strip, and they must
  occur in conjugate pairs. Suppose we have a double real
  root, then
  \begin{eqnarray}
    \label{eq:15}
    \sin m\pi z &=& \beta\sin\pi z
    \\\na{and}\label{eq:16}
    m\cos m\pi z &=& \beta\cos\pi z
  \end{eqnarray}
Squaring both (\ref{eq:15}) and (\ref{eq:16}), and
      adding the results gives
\begin{eqnarray}
\nn
    &\sin^{2}m\pi z + \cos^{2} m\pi z
    + (m^{2}-1)\cos m\pi z = \beta^{2}&
    \\\na{or}\label{eq:17}
    &\cos^{2} m\pi z = \frac{\beta^{2}-1}{m^{2}-1}&
  \end{eqnarray}
  which is impossible for \(\beta>m\) because the right
  hand side is bigger than one. If \(\beta=m\), then
  (\ref{eq:17}) and (\ref{eq:16}) are possible, but only
  if
\begin{eqnarray}
  \nn
  \cos \pi z = \cos m \pi z = \pm 1
\end{eqnarray}
which implies that  \(z\) and  \(mz\) are either both even
or both odd integers. This is equivalent to the statement
that  \(z=\jm\) and \jpfjm\ is an even integer. One can
check directly that the root is indeed triple in this case.
\end{proof}
\begin{proof}[Proof of Theorem \ref{sec:roots-p_m-p}]

We only prove the theorem for \Pmp, as the proof for
\Pmm\ is analogous.
  Item \ref{item:6} follows immediately from item
  \ref{item:4} of corollary \ref{sec:roots-all-beta}. Item
  \ref{item:4} of corollary \ref{sec:roots-all-beta} also
  implies all of item \ref{item:7}, except the fact that
  the two roots are not real. To see this, note that, as
  \(\beta\rightarrow\infty\), the roots of \(\pb\) must
  either approach the roots of \(\sin\pi z\), which are
  integers, or leave every compact subset of the complex
  plane. Since \ijm contains no integers, the two roots in
  \sjm must leave every compact subset of \C . As
  \(\real{z}\) must remain between \jm\  and \jpm\ , the
  imaginary parts of these two roots must grow unbounded
  as  \(\beta\rightarrow\infty\). If they were real at
  \(\beta=m\), lemma \ref{sec:beta-roots-simple} would
  imply that they were simple and remained simple and real
  for all  \(\beta\), contradicting the fact that their
  imaginary parts grow unbounded as
  \(\beta\rightarrow\infty\). Hence they must not have
  been real at  \(\beta=m\).\\

Item \ref{item:8} follows for a similar reason. We can
check directly that \jm\ is either a simple or a triple
root, depending on the parity of \jpfjm. Item \ref{item:5}
of corollary \ref{sec:roots-all-beta} tells us that these
are all of the roots in the closure of
\(S_{\jmm}\cup\sjm\), so there are no other roots of
\Pmp in these strips. 
\end{proof}

\section{Some Limiting Cases and Conclusions}
\label{sec:some-limiting-cases}

We have shown that there are two transmission wavenumbers
in each closed strip
\(\frac{2j\pi}{(\sigma+1)L}\le\real{k}\le\frac{2(j+1)\pi}{(\sigma+1)L}\)
of width \(\frac{2\pi}{(\sigma+1)L}\), and two real
transmission wavenumbers in each of those strips which
contains an integer.
We show below that, in the limiting cases of a weak scatterer
(\(\sigma\rightarrow 1\)), there are no longer any real
transmission wavenumbers, and in the limiting case of a
strong scatterer (\(\sigma\rightarrow\infty\)), all the
transmission wavenumbers become real.\\

The Born, or linear approximation, is the limit as
\(\sigma\rightarrow 1\). We return to equation
(\ref{eq:6}), set \(\pi z=kL\), and let
\(\sigma\rightarrow 1\) to obtain

\begin{eqnarray}
  \label{eq:23}
  \pi z = \pm\sin \pi z
\end{eqnarray}

Introducing the parameter  \(\beta\) in front of the
\(\pi z\) on the right hand side, and repeating the steps
we used to prove   theorem \ref{sec:roots-p_m}, gives:

\begin{theorem}
  For every integer \(j\) except \(j=-1\) and \(j=0\),
  equations (\ref{eq:23}) have exactly two non-real
  complex conjugate roots in the strip \(j\le\real{j}\le
  j+1\), and they lie in the interior of the strip.  The
  only root in \(-1\le\real{z}\le 1\) is the triple root
  at \(z=0\).
\end{theorem}

The imaginary parts of these roots grow logarithmically.
The precise asymptotics, which can be verified by
substituting (\ref{eq:5}) into (\ref{eq:23}),  are
  \begin{eqnarray}
    \label{eq:5}
    &z_{j}^{\pm} = (j+\half ) \pm i\log(2|j+\half|) +
    O\left(\frac{\log(2|j+\half|)}{(j+\half)}\right)&
    \\\na{so that \(k_{j} =  \frac{\pi z_{j}}{L}\) satisfy\espace }\nn
        &k_{j}^{\pm} \sim\frac{\pi}{L}\left( (j+\half) \pm
       i\log(2|j+\half|)\right)&
    \end{eqnarray}
    where, for the odd integers \(j\), \(z_{j}\) solve
    equation (\ref{eq:23}) with the plus sign, and, for
    even integers \(j\), the \(z_{j}\) solve (\ref{eq:23})
    with the minus sign.\\

    The limit as \(\sigma\rightarrow\infty\) is a little
    more work to calculate. This is the limit as
    \(m\rightarrow 1\), so that equation (\ref{eq:11})
    becomes
\begin{eqnarray}
  \nn
  \sin\pi z &=& -\sin\pi z
  \\\na{or}
  \sin\pi z &=& 0
\end{eqnarray}
with roots at the integers. Thus, as
\(\sigma\rightarrow\infty\), the even transmission wavenumbers
\begin{eqnarray}
  \nn
  k_{j}\sim \frac{2\pi j}{\sigma L}
\end{eqnarray}
have the same asymptotics as the square roots of the
eigenvalues of the Dirichlet problem
\begin{eqnarray}
  \nn
   &u'' + k^{2}\sigma^{2}u =0&
   \\
   &u(\pm\halfL)=0&
\end{eqnarray}
with odd (\(u(x)=-u(-x)\)) eigenfunctions.
\\

The solutions of (\ref{eq:10}), which describe the
transmission wavenumbers with odd eigenfunctions, have a
different asymptotic behavior, that will match the
asymptotics of a different self-adjoint boundary value
problem. To see the limiting equation, we subtract
\(\sin\pi z\) from both sides of (\ref{eq:10})
  \begin{eqnarray}
    \nn
    \sin(m\pi z) - \sin\pi z &=& m\sin\pi z - \sin\pi z
    \\\na{\ni and divide by  \(m-1\)\espace}\nn
    \frac{\sin(m\pi z) - \sin\pi z }{m-1} &=& m\sin\pi z
    \\\na{\ni which becomes, as  \(m\rightarrow 1\)}\nn
    \pi z \cos\pi z = \sin\pi z
    \\\na{or}\nn
    \pi z = \tan\pi z
  \end{eqnarray}
   The transmission wavenumbers satisfy

   \begin{eqnarray}
     \label{eq:4}
     k(\sigma-1)\halfL = \tan k(\sigma-1)\halfL
   \end{eqnarray}
   and thus have the same asymptotics, as
   \(\sigma\rightarrow\infty\), as the square roots of the
   eigenvalues of the self-adjoint boundary value problem.
\begin{eqnarray}
  \label{eq:45}
  &v'' + k^{2}\sigma^{2}v = 0&
\\\nn
&\pm\halfL v'(\pm\halfL) = v(\pm\halfL)
\end{eqnarray}
with odd eigenfunctions\footnote{The wavenumbers of
  (\ref{eq:45}) with even
eigenfunctions satisfy equation (\ref{eq:4}) with tangent
replaced by cotangent}.\\

Because these limiting wavenumbers are the square roots of
the eigenvalues of a self-adjoint problem, we see that, as
\(\sigma\rightarrow\infty\), all transmission wavenumbers
become real.\\

In our discussion of the one dimensional transmission
eigenvalue problem, we have observed certain relationships
between transmission eigenvalues, Dirichlet eigenvalues,
and the eigenvalues of another self-adjoint boundary value
problem. We expect (optimistically) that some of these
relationships will persist in higher dimensions. We know
that, in the Born approximation, there are no real
transmission eigenvalues \cite{MR2438784}, that a generic
radially symmetric index of refraction has complex
transmission eigenvalues \cite{MR2944956}, and that there
are infinitely many transmission eigenvalues
\cite{MR2438784} \cite{MR2596553}. All of these results
are consistent with our one dimensional conclusions.\\

A consequence of theorem \ref{sec:itw-main} is that the
counting function for the number of interior transmission
wavenumbers in a strip of width  \(W\) is
\begin{eqnarray}
  \nn
  N_{\C}(W;\sigma,L) =  \frac{(\sigma+1)LW}{\pi} \pm 2
\end{eqnarray}
which has the same asymptotics as the counting function
for Dirichlet wavenumbers for an interval of the same
length with index of refraction  \(n=\sigma+1\).\\

The counting function for real interior transmission
wavenumbers in an interval of width  \(W\) is 
\begin{eqnarray}
  \nn
  N_{\R}(W;\sigma,L) =  \frac{(\sigma-1)LW}{\pi} \pm 2
\end{eqnarray}
which matches the asymptotics of the counting function for
Dirichlet wavenumbers for an interval of the same length
with index of refraction \(n=\sigma-1\).\footnote{This is
  correct as is for irrational \(\sigma\). For rational
  \(\sigma\) we need to adopt the convention that every
  quadruple root (there will always be some)  is  counted
  as two real roots and two non-real roots.}
It seems reasonable to conjecture that the relationship
between the counting functions for the Dirichlet wavenumbers
and the transmission wavenumbers remains true for non-constant
\(\sigma\), and in higher dimensions
as well.

\nocite{MR983987,MR934695,MR1019312,MR2262743,MR2438784,MR2596553,MR2888291,MR2944956}
\bibliographystyle{plain}
\bibliography{cit}
\end{document}